\documentclass[12pt]{amsart}
\usepackage{amscd,amssymb,graphics}

\usepackage{amsfonts}
\usepackage{amsmath}
\usepackage{amsxtra}
\usepackage{latexsym}
\usepackage[mathcal]{eucal}

\usepackage{graphics,colortbl}

\input xy
\xyoption{all}
\usepackage{epsfig}

\usepackage[pdftex,bookmarks,colorlinks,breaklinks]{hyperref}

\newtheorem{theorem}{Theorem}[section]
\newtheorem{fact}[theorem]{Fact}
\newtheorem{corollary}[theorem]{Corollary}
\newtheorem{lemma}[theorem]{Lemma}
\newtheorem{proposition}[theorem]{Proposition}

\newtheorem{definition}[theorem]{Definition}
\numberwithin{equation}{section}

\theoremstyle{remark}
\newtheorem{remark}[theorem]{Remark}

\newcommand{\ben}{\begin{enumerate}}
	\newcommand{\een}{\end{enumerate}}
\newcommand{\bit}{\begin{itemize}}
	\newcommand{\eit}{\end{itemize}}

\def\R {{\Bbb R}}
\def\Q {{\Bbb Q}}
\def \F {{\Bbb F}}

\def\C {{\Bbb C}}
\def\N{{\Bbb N}}

\def\Z {{\Bbb Z}}

\def\F{{\Bbb F}}

\def\QED{\nobreak\quad\ifmmode\roman{Q.E.D.}\else{\rm Q.E.D.}\fi}

\def\PSL{\operatorname{PSL}}

\def\SL{\operatorname{SL}}
\def\UT{\operatorname{UT}}

\def\STP{\operatorname{ST^+}}
\begin{document}
	
	\title[]{Minimality conditions equivalent  to  the finitude of Fermat  and Mersenne primes}
	\author[]{Menachem Shlossberg}
	\address[M. Shlossberg]
	{\hfill\break School of Computer Science
		\hfill\break Reichman University, 4610101 Herzliya 
		\hfill\break Israel}
	\email{menachem.shlossberg@post.idc.ac.il}
	
	\maketitle
	
	\setcounter{tocdepth}{1}
	\begin{abstract}
It is still open whether there exist infinitely many  Fermat primes or infinitely many  composite Fermat  numbers. The same question concerning the Mersenne numbers is  also unsolved. Extending some results from \cite{MS}, we
characterize  the Fermat primes and the Mersenne primes in terms of  topological minimality of some matrix groups.
This is done by showing, among other things,  that if  $\F$ is a  subfield of a local field of characteristic $\neq 2$, then   the special upper triangular group $\STP(n,\F)$ is minimal precisely when  the special linear group $\SL(n,\F)$ is.  We provide criteria for the minimality (and total minimality) of $\SL(n,\F)$ and  $\STP(n,\F),$     where  $\F$ is a  subfield of $\C.$   

Let $\mathcal F_\pi$ and  $\mathcal F_c $ be the set  of Fermat primes and the set of  composite Fermat numbers, respectively. As our main result, we prove that the following conditions are equivalent for   $\mathcal A\in\{\mathcal F_\pi, \mathcal F_c\}:$
	\bit
	 \item  $\mathcal A$   is finite;
		\item  $\prod_{F_n\in \mathcal A}\SL(F_n-1,\Q(i))$ is minimal,   where $\Q(i)$ is the Gaussian rational field;
		\item  $\prod_{F_n\in \mathcal A}\STP(F_n-1,\Q(i))$ is minimal.
		\eit 
	Similarly, denote by $\mathcal M_\pi$ and  $\mathcal M_c $  the set  of   Mersenne primes and the set of  composite Mersenne numbers, respectively, and let $\mathcal B\in\{ \mathcal M_\pi, \mathcal M_c\}.$ Then the following conditions are equivalent:
	\bit
		 \item  $\mathcal B$   is finite;
		\item $\prod_{M_p\in \mathcal B}\SL(M_p+1,\Q(i))$ is minimal;
\item $\prod_{M_p\in \mathcal B}\STP(M_p+1,\Q(i))$ is minimal.
\eit
		
	\end{abstract}

	\section{Introduction}
	A Fermat number has the form $F_n=2^{2^n} +1,$ where $n$ is a non-negative integer while a 
	Mersenne number has the form  $M_p=2^p-1$ for some prime $p.$ Note that $2^n-1$ is composite when $n$ is composite. In other words, a
	Mersenne prime is a prime number that is one less than a power of two.
	There are several  open problems concerning these numbers (e.g., see \cite{G}).   For example, it is still unknown whether the Fermat primes,  the composite Fermat numbers, the Mersenne primes  or the composite Mersenne numbers are infinitely many. 

	
	 All topological groups in this paper are Hausdorff. Let $\F$ be a topological subfield of a local field. Recall that a \emph{local field} is a non-discrete locally compact topological field. 
	 Denote by $\SL(n,\F)$  the special linear group  over  $\F$ of degree $n$  equipped with the 
	  pointwise topology 
	inherited from $\F^{n^2},$ and by $\STP(n,\F)$ its topological subgroup  consisting of upper triangular matrices. 
Megrelishvili and the author  characterized the  Fermat primes in terms of the topological minimality of some special linear groups. Recall that a  
topological group $G$ is  \emph{minimal} \cite{Doitch,S71} if 
every continuous isomorphism $f \colon G\to H$, with $H$ 
a topological group, is a topological isomorphism 	(equivalently, if $G$ does not admit a strictly coarser group topology).  \begin{theorem} \cite[Theorem 5.5]{MS}\label{thm:fermat}
		For an odd prime  $p$ the following conditions are equivalent: \ben\item $p$ is a  Fermat prime;
		\item $\SL(p-1,(\Q,\tau_p))$ is minimal, where $(\Q,\tau_p)$ is the field of rationals equipped with the $p$-adic topology;
		\item  $\SL(p-1,\Q(i))$ is minimal, where $\Q(i) \subset \C$ is the Gaussian rational field. 
		\een  
	\end{theorem}
A  similar characterization	for the Mersenne primes is provided in Theorem \ref{thm:charofmer}.
Note that	it follows from 
Gauss--Wantzel Theorem  that an odd prime $p$ is a Fermat prime if and only if a $p$-sided regular polygon can be constructed with  compass and  straightedge.

 We prove in Theorem \ref{thm:minifmin} that if  $\F$ is a  subfield of a local field of characteristic distinct than $2$, then   the special upper triangular group $\STP(n,\F)$ is minimal if and only if  the special linear group $\SL(n,\F)$ is minimal. This result with some other tools yield criteria for the minimality (and total minimality) of $\SL(n,\F)$ and  $\STP(n,\F),$     where  $\F$ is a  subfield of $\C$ (see Proposition \ref{ex:slnotmin}, Remark \ref{rem:last} and  Corollary \ref{cor:alsostp}).
 
As a main result,  we prove in Theorem \ref{thm:main} that the finitude of  Fermat and Mersenne primes as well as the finitude of  composite  Fermat and Mersenne numbers is equivalent to the minimality of some topological products of some matrix groups.
%
	
\section{Minimality of $\STP(n,\F)$ and $\SL(n,\F)$}
	Let $\mathsf{N}:=\UT(n,\F)$ and $\mathsf{A}$ be the subgroups of $\STP(n,\F)$     consisting of upper unitriangular matrices and diagonal matrices, respectively.  Note that $\mathsf{N}$ is normal in $\STP(n,\F)$ and  $\STP(n,\F)\cong \mathsf{N} \rtimes_{\alpha} \mathsf{A}$, where  $\alpha$ is the action by conjugations. It is known that  $\mathsf{N}$ is the derived subgroup of  $\STP(n,\F).$ 
Recall also that $\SL(n,\F)$ has finite center (e.g., see \cite[3.2.6]{RO}) 
\[Z(\SL(n,\F))=\{\lambda I_n:\lambda\in \mu_n\},\] where $\mu_n$ is a finite group  consisting of the $n$-th roots of unity in $\F$ and $I_n$ is the identity matrix of size $n.$ 

\begin{lemma}\label{lemma:samec}
	Let $\F$ be a field and $n\in \N.$ Then 
	$Z(\STP(n,\F))=Z(\SL(n,\F))$ and $\STP(n,\F)/Z(\STP(n,\F))$ is center-free.
\end{lemma}
\begin{proof}
 Let $(C,D)\in Z(\mathsf{N} \rtimes \mathsf{A})$ and $E\in \mathsf{A}$ such that
	$e_{ii}\neq e_{jj}$ whenever $i\neq j.$  For every $i>j$ it holds that 
	\[(CE)_{ij}=\sum_{t=1}^n c_{it}e_{tj}=c_{ij}e_{jj}\]
	and \[(EC)_{ij}=\sum_{t=1}^n e_{it}c_{tj}=c_{ij}e_{ii}\]
	Then, the equality $(C,D)(I,E)=(I,E)(C,D)$ implies that $CE=EC$ and $c_{ij}e_{jj}=c_{ij}e_{ii}$. As $e_{ii}\neq e_{jj}$ we deduce that $c_{ij}=0.$ Since 
	$C$ is an upper unitriangular matrix it follows that $C=I.$ To prove that $Z(\STP(n,\F))=Z(\SL(n,\F))$ it suffices to show that the diagonal matrix $D$ is  scalar. To this aim, pick distinct indices $i,j$ and a matrix $F\in \mathsf{N}$ such that $f_{ij}\neq 0.$ As $(I,D)\in Z(\mathsf{N} \rtimes \mathsf{A}),$ it follows that $(I,D)(F,I)=(F,I)(I,D).$ This implies that
	$DF=FD$ and, in particular, $(DF)_{ij}=(FD)_{ij}.$ This yields the equality $d_{ii}f_{ij}=d_{jj}f_{ij}$ since $D$ is diagonal. We conclude that $d_{ii}=d_{jj}$ in view of the inequality  $f_{ij}\neq 0.$ This proves that  $Z(\STP(n,\F))=Z(\SL(n,\F)).$\\
	Now, let $(B,D)Z(\mathsf{N} \rtimes \mathsf{A})\in Z(\mathsf{N} \rtimes \mathsf{A}/ Z(\mathsf{N} \rtimes \mathsf{A}))$ and  $(C,E)\in \mathsf{N} \rtimes \mathsf{A}.$ By what we proved,  there exists a scalar $\lambda\in \F$ such that $(B,D)(C,E)=(C,E)(B,D)(I,\lambda I).$
	Therefore, $DE=\lambda DE$ and $\lambda=1.$ This means that $(B,D)(C,E)=(C,E)(B,D)$ for every $(C,E)\in \mathsf{N} \rtimes \mathsf{A}.$ Therefore, $(\mathsf{N} \rtimes \mathsf{A})/ Z(\mathsf{N} \rtimes \mathsf{A})$ and its isomorphic copy $\STP(n,\F)/Z(\STP(n,\F))$ are center-free.
\end{proof}
The following lemma will be useful in proving Theorem \ref{thm:minifmin}.
\begin{lemma}\label{lem:ngeq3}
	Let $\F$ be a subfield of a  field $H$ and	let $n\geq 3$ be a natural number. If $L$ is a normal subgroup of  $\STP(n,H)$ that intersects $\UT(n,H)$ non-trivially, then it intersects $\UT(n,\F)$ non-trivially.
\end{lemma}
\begin{proof}
	Since $L\cap \UT(n,H) $  is a non-trivial normal subgroup of the nilpotent group $\UT(n,H)$  it must non-trivially intersects the center $Z(\UT(n,H)). $  
	Then there exists \[I\neq B=\left(\begin{array}{ccccc} 
		1 & 0 & \ldots &  0 &	b \\
		0 & 1 & \ddots  & \vdots &0\\
		\vdots & \ddots & \ddots &\ddots & \vdots \\
		\vdots & \ldots & 0 & 1 &0\\
		0 & \ldots & \ldots & 0 & 1
	\end{array} \right)\in L\cap Z(\UT(n,H)) \] for some $b\in H$  (see \cite[p. 94]{SU} for example).   Since $n\geq 3$ there exists a diagonal matrix $D\in \STP(n,H)$  such that $d_{11}=b^{-1}$ and $d_{nn}=1$. This implies that \[I\neq DBD^{-1}\in  Z(\UT(n,\F)). \qedhere\] 
\end{proof}
\begin{definition}
	Let $H$ be a subgroup of a topological group $G$. Then $H$ is   essential in $G$ if $H\cap L \neq \{e\}$ 
	for every non-trivial closed normal subgroup $L$ of $G$. 
\end{definition}
The following minimality criterion of dense subgroups is well-known (for compact $G$ see also \cite{P,S71}).
\begin{fact} \label{Crit} \cite[Minimality Criterion]{B}  
	Let $H$ be a dense subgroup of a topological group $G.$  Then $H$ is minimal if and only if $G$ is minimal and $H$ is essential in $G$. 
\end{fact} 
\begin{remark}\label{remark:forlocal}
	If $\F$ is a subfield of a local field $P,$ then its
	completion $\widehat\F$  is a topological field that can be identified with the closure of $\F$ in $P$. In case $\F$ is  infinite then  $\widehat\F$ is also a local field, as
	the local field $P$ contains no infinite discrete subfields (see \cite[p.  27]{MA}).
\end{remark}

\begin{proposition}\label{pro:minsl}\cite[Proposition 5.1]{MS}
	Let $\F$ be a 
	subfield of a local field. 
	Then the following conditions are equivalent: \ben \item  $\SL(n,\F)$ is   minimal; \item any non-trivial central subgroup of $\SL(n,\widehat{\F})$  intersects $\SL(n,\F)$ non-trivially  (i.e., if $1\neq\lambda\in \mu_n(\widehat{\F})$, then there exists $k\in \Z$ such that $1\neq \lambda^k\in \F$). \een
\end{proposition}
\begin{theorem}\label{thm:minifmin}
Let $\F$ be a  subfield of a local field of characteristic distinct than $2.$ Then, $\SL(n,\F)$ is   minimal if and only if $\STP(n,\F)$ is minimal.
\end{theorem}
\begin{proof} Without loss of generality we may assume that $\F$ is infinite. Suppose first that $\STP(n,\F)$ is minimal.  By Lemma \ref{lemma:samec},  $Z(\STP(n,\F))=Z(\SL(n,\F))$. Since this center is finite it follows from the Minimality Criterion that any non-trivial central subgroup of $\STP(n,\widehat{\F})$  intersects $\STP(n,\F)$ non-trivially.  This implies that any non-trivial central subgroup of $\SL(n,\widehat{\F})$  intersects $\SL(n,\F)$ non-trivially. By Proposition \ref{pro:minsl}, $\SL(n,\F)$ is   minimal.
	
Conversely,  let us assume that $\SL(n,\F)$ is   minimal. In case $n=2$ then  $\STP(n,\F)$ is minimal by \cite[Theorem 3.4]{MS} as an infinite subfield of a local field is   locally retrobounded and non-discrete. So, we may assume that $n\geq 3.$  By \cite[Theorem 3.19]{MS}, $\STP(n,\widehat{\F})$ is minimal as $\widehat{\F}$ is a local field (see Remark \ref{remark:forlocal}).   In view of the Minimality Criterion, it suffices to show that  $\STP(n,\F)$ is essential in $\STP(n,\widehat{\F})$. Let $L$ be closed normal non-trivial subgroup of $\STP(n,\widehat{\F})$. If \[L\subseteq Z(\STP(n,\widehat{\F}))=Z(\SL(n,\widehat{\F})),\]  then  $L$ intersects $\SL(n,\F)$ non-trivially by Proposition \ref{pro:minsl}. Clearly, this implies that $L$ intersects $\STP(n,\F)$ non-trivially. If $L$ is not central, then it must non-trivially intersect  $\UT(n,\widehat{\F}),$  the derived subgroup of $\STP(n,\widehat{\F})$, in view of  \cite[Lemma 2.3]{XDSD}. Now,  Lemma \ref{lem:ngeq3}  implies that $L$ intersects $\STP(n,\F)$ non-trivially and we deduce that $\STP(n,\F)$ is essential in $\STP(n,\widehat{\F})$.
\end{proof}
In view of Theorem \ref{thm:fermat} and Theorem \ref{thm:minifmin},  the following characterization of the the  Fermat primes is obtained.

\begin{theorem} \label{thm:newfermat}
	For an odd prime  $p$ the following conditions are equivalent: \ben\item $p$ is a  Fermat prime;
	\item $\STP(p-1,(\Q,\tau_p))$ is minimal;
	\item  $\STP(p-1,\Q(i))$ is minimal. 
	\een  
\end{theorem}

The following concept has a key role in the Total Minimality Criterion. 
\begin{definition}
	A subgroup $H$ of a topological group $G$ is totally dense if for every closed normal subgroup
	$L$ of $G$ the intersection $L \cap H$ is dense in $L.$	
\end{definition} 
\begin{fact}\label{fact:tmc} \cite[Total Minimality Criterion]{DP}
	Let $H$ be a dense subgroup of a topological group $G$. Then $H$ is
	totally minimal if and only if $G$ is totally minimal and $H$ is totally dense in $G$.
\end{fact}
\begin{theorem}\label{t:SLSMALL} \cite[Theorem 4.7]{MS}
	Let $\F$ be a 
	subfield of a local field.
	Then 
	$\SL(n,\F)$ is totally minimal if and only if
	$Z(\SL(n,\F))=Z(\SL(n,\widehat{\F}))$  (i.e., $\mu_n(\F)=\mu_n(\widehat{\F})$).
\end{theorem}

Let $\rho_m=e^{\frac{2\pi i }{m}}$ be the $m$-th primitive root of unity. The  next result extends \cite[Corollary 5.3]{MS}, where $\rho_4=i$ is considered. 
\begin{proposition}\label{ex:slnotmin}
	Let $\F$ be a dense subfield of $\C.$  Then, \ben \item   $\SL(n,\F)$ is totally minimal if and only if  $\rho_n\in \F;$  
	\item   $\SL(n,\F)$ is minimal if and only if     $\langle \rho_m \rangle \cap \F$ is non-trivial whenever $m$ divides $n.$  
	\een
\end{proposition}
\begin{proof}
	(1) Necessity: follows from Theorem \ref{t:SLSMALL}.  Indeed, $\lambda=\rho_n\in \C$ is an $n$-th root of unity.
	
	Sufficiency: if $\lambda\in \C $ and $\lambda^n=1$,
	then $\lambda\in \langle \rho_n \rangle\subseteq \F.$ So, we may use  Theorem \ref{t:SLSMALL} again.
	
	(2) Necessity: Let $1\neq \lambda\in \C$ be an $n$-th root of unity. Then $\lambda$ is an $m$-th primitive root of unity where $m$ divides $n.$ Since $\SL(n,\F)$ is minimal it follows that there exists $k$ such that
	$1\neq \lambda^k\in \F.$ Clearly, $\lambda^k\in \langle \rho_m \rangle \cap \F.$  So, $\langle \rho_m \rangle \cap \F$ is non-trivial. Now use Proposition \ref{pro:minsl}.

	Sufficiency: Let $1\neq \lambda\in \C$ be an $n$-th root of unity.   Then $\lambda$ is an $m$-th primitive root of unity where $m$ divides $n.$ This means that 
	$\lambda=e^{\frac{2\pi i k}{m}}=(\rho_m)^k$, where $1\leq k\leq m$  with $\gcd(k,m)=1.$
	By our assumption,  $\langle \rho_m \rangle \cap \F$ is non-trivial. Hence, there exists $l$ such that $1\neq (\rho_m)^l\in\langle \rho_m \rangle \cap \F.$ Since $\gcd(k,m)=1$ and 	$\lambda=(\rho_m)^k$, it follows that there exists $t\in \Z$ such that $(\rho_m)^l=\lambda^t.$ This proves the minimality of   $\SL(n,\F),$ in view of  Proposition \ref{pro:minsl}.
\end{proof}
\begin{remark}\label{rem:last}
It is known that  a  subfield $\F$ is dense in $\C$ if and only if it is not contained in $\R.$ By \cite[Corollary 4.8]{MS}, if $\F\subseteq \R,$ then $\SL(n,\F)$ is totally minimal for every $n\in \N$. So, together with Proposition \ref{ex:slnotmin} we obtain criteria for the minimality and total minimality of $\SL(n,\F),$ where $\F$ is any subfield of $\C$ and  $n\in \N.$ 
\end{remark}
Since $\C$ has zero characteristic, Theorem \ref{thm:minifmin}, Proposition \ref{ex:slnotmin} and Remark \ref{rem:last}  imply:
\begin{corollary}\label{cor:alsostp}
Let $\F$ be a topological subfield of $\C.$ \ben \item 
If $\F$ is dense in $\C$, 
 then $\STP(n,\F)$ is minimal if and only if     $\langle \rho_m \rangle \cap \F$ is non-trivial whenever $m$ divides $n.$
 \item If $\F\subseteq \R$,  then $\STP(n,\F)$ is  minimal for every $n\in \N$.\een
\end{corollary}
\section{Proof of the main result}

 By \cite[Corollary 5.3]{MS}, $\SL(n,\Q(i))$ is minimal if and only if $n=2^k,$ where $k$ is a non-negative integer.  This immediately implies  the following theorem  concerning the Mersenne primes (compare with  Theorem \ref{thm:fermat} and Theorem \ref{thm:newfermat}).
 \begin{theorem}\label{thm:charofmer} 
 	For a prime  $p$ the following conditions are equivalent:
 	\ben
 	\item $p$ is a Mersenne prime;
 	\item $\SL(p+1,\Q(i))$ is minimal;
 	\item  $\STP(p+1,\Q(i))$ is minimal.
 	\een
 \end{theorem}
\begin{theorem}\label{thm:pofsl}
\
\ben \item	If $\F$ is a local 
	field, then  $\prod_{n\in \N}\SL(n,\F)$   is minimal. \item If, in addition, $\mathrm{char}(\F)\neq 2,$ then $\prod_{n\in \N}\STP(n,\F)$ is minimal.
	\een
\end{theorem}
\begin{proof}
(1) Since a compact group is minimal we may assume without loss of generality that $\F$ is  infinite.	By \cite{BG} (see also \cite[Theorem 4.3]{MS}), the projective special linear group $\PSL(n,\F)=\SL(n,\F)/Z(\SL(n,\F))$ (equipped with the quotient topology) is minimal for every $n\in \N.$ Being algebraically simple (see \cite[3.2.9]{RO}), $\PSL(n,\F)$ has trivial center.  Therefore, the topological product $\prod_{n\in \N}\PSL(n,\F)$ is minimal by \cite[Theorem 1.15]{MEG95}. As \[\prod_{n\in \N}\PSL(n,\F)\cong \prod_{n\in \N}\SL(n,\F)/Z\Big(\prod_{n\in \N}\SL(n,\F)\Big),\] where $Z\Big(\prod_{n\in \N}\SL(n,\F)\Big)$ is compact,  it follows from \cite[Theorem 7.3.1]{DPS89} that 

 $\prod_{n\in \N}\SL(n,\F)$ is minimal.\\
 (2) By Lemma \ref{lemma:samec}, 
  the center of $\STP(n,\F)/Z(\STP(n,\F))$ is trivial for every $n\in \N$, where	$Z(\STP(n,\F))=Z(\SL(n,\F)).$   
  By \cite[Theorem 3.17]{MS},   $\STP(n,\F)/Z(\STP(n,\F))$ is minimal.  We complete the proof using the topological isomorphism \[\prod_{n\in \N}\STP(n,\F)/\prod_{n\in \N}Z(\STP(n,\F))\cong \prod_{n\in \N}(\STP(n,\F)/Z(\STP(n,\F)))\] and similar arguments to those appearing in the proof of (1).
\end{proof}
\begin{definition}\cite{St1} 	A minimal group $G$ is perfectly minimal if $G\times H$ is minimal for every minimal group $H.$
\end{definition}

	\begin{proposition}\label{prop:p2isper}
Let $\F$ be a subfield of a local field. 
Then $\SL(2^k,\F)$ is perfectly minimal for every $k\in \N.$ If $\mathrm{char}(\F)\neq 2,$ then	 $\STP(2^k,\F)$ is perfectly minimal for every $k\in \N.$
	\end{proposition}
\begin{proof}
Let $\F$ be a subfield of a local field and $k\in \N.$ By \cite[Corollary 5.2]{MS}, $\SL(2^k,\F)$ is minimal.
Being finite, the center $Z(\SL(2^k,\F))$ is perfectly minimal (see  \cite{Doitch}). Having perfectly minimal center,  the minimal group $\SL(2^k,\F)$ is perfectly minimal in view of \cite[Theorem 1.4]{MEG95}. The last assertion is proved similarly, taking into account that $Z(\SL(2^k,\F))=Z(\STP(2^k,\F))$ and the fact that $\STP(2^k,\F)$ is minimal by Theorem \ref{thm:minifmin}.
\end{proof}

	\begin{theorem}\label{prop:nminprod}
		Let $(n_k)_{k\in \N}$ be an increasing  sequence of natural numbers.
		Then, neither $\prod_{k\in \N}\SL(2^{n_k},\Q(i))$ nor  $\prod_{k\in \N}\STP(2^{n_k},\Q(i))$ are minimal.
	\end{theorem}
	\begin{proof}
We  first prove that	$G=\prod_{k\in \N}\SL(2^{n_k},\Q(i))$ is not minimal. In view of the Minimality Criterion, it suffices to show that $G$ is not essential in  $\widehat G=\prod_{k\in \N}\SL(2^{n_k},\C).$  
		To this aim, let \[N=\{ (\lambda_kI_{2^{n_k}})_{k\in \N}\in \widehat G | \ (\lambda_{k+1})^2=\lambda_{k} \ \forall k\in \N \}.\] 
	
		The equality $\lambda_{k+1}^2=\lambda_{k}$ implies  that  $N$ is a closed central subgroup of $\widehat G.$ Moreover, $N$ is non-trivial as the sequence $(n_k)_{k\in \N}$ is increasing.   Let us see that $N$ trivially intersects $G.$ Otherwise, there exists a sequence $(\lambda_k)_{k\in \N}$ of roots of unity in $\Q(i)$ such that  $(\lambda_{k+1})^2=\lambda_{k}$ for every $k\in \N$  and $\lambda_{k_0}\neq 1$ for some $k_0\in \N.$ It follows that
		$\lambda_{k_0},\lambda_{k_0+1},\lambda_{k_0+2},\lambda_{k_0+3}$ are  different non-trivial roots of unity in $\Q(i),$ contradicting the fact that $\pm 1,\pm i$ are the only roots of unity in $\Q(i).$ 
		
	Now consider the group	$H=\prod_{k\in \N}\STP(2^{n_k},\Q(i))$. In view of Lemma \ref{lemma:samec} and what we have just proved, $N$ is also a closed non-trivial central subgroup of $\widehat H$  that trivially intersects $H.$ This means that $H$ is not essential in $\widehat H$. By the Minimality Criterion, $H$ is not minimal.
	\end{proof}
\newpage
\begin{theorem}\label{thm:main}\
\begin{enumerate}	
	\item Let $\mathcal F_\pi$ and  $\mathcal F_c $ be the set  of Fermat primes and the set of  composite Fermat numbers, respectively, and let  $\mathcal A\in\{ \mathcal F_\pi, \mathcal F_c\}.$ Then the following conditions are equivalent: \begin{enumerate}  \item $\mathcal A$   is finite;
\item	$\prod_{F_n\in \mathcal A}\SL(F_n-1,\Q(i))$ is minimal;
\item  $\prod_{F_n\in \mathcal A}\STP(F_n-1,\Q(i))$ is minimal.
\end{enumerate} 
	\item  Let $\mathcal M_\pi$ and  $\mathcal M_c $ be the set  of   Mersenne primes and the set of  composite Mersenne numbers, respectively, and let  $\mathcal B\in\{ \mathcal M_\pi, \mathcal M_c\}.$
	Then the following conditions are equivalent: \ben \item $\mathcal B$   is finite; 
	\item $\prod_{M_p\in \mathcal B}\SL(M_p+1,\Q(i))$ is minimal;
	\item $\prod_{M_p\in \mathcal B}\STP(M_p+1,\Q(i))$ is minimal.
	\een 
\end{enumerate}
\end{theorem}
\begin{proof}
(1) Assume first that $\mathcal A$ is finite. Note that $F_n-1$ is a power of two for every Fermat number $F_n\in \mathcal A.$ It is easy to see that a product of finitely many perfectly minimal groups is minimal. Therefore, both topological groups $\prod_{F_n\in \mathcal A}\SL(F_n-1,\Q(i))$
and $\prod_{F_n\in \mathcal A}\STP(F_n-1,\Q(i))$ are minimal, in view of Proposition \ref{prop:p2isper}.
If $\mathcal A$ is infinite, then  $\prod_{F_n\in \mathcal A}\SL(F_n-1,\Q(i))$ and  $\prod_{F_n\in \mathcal A}\STP(F_n-1,\Q(i))$ are not minimal by Theorem \ref{prop:nminprod}. \\
(2) The proof is similar to (1). The only difference is that for every prime $p$  it holds that $M_p+1$ is a power of two.
\end{proof}
	\bibliographystyle{plain}

\end{document}